\documentclass[11pt]{article}
\usepackage[T1]{fontenc}
\usepackage{amsfonts}
\usepackage{amsmath}
\usepackage{amssymb}
\usepackage{amsthm}
\usepackage{bbm}
\usepackage{bm}
\usepackage{mathrsfs}
\usepackage{verbatim}
\usepackage{setspace}
\usepackage{color}
\usepackage{pdfsync}
\usepackage{enumitem}
\usepackage{graphicx}
\usepackage{subfigure}
\usepackage{tikz}
\usetikzlibrary{patterns}

\theoremstyle{plain}
\newtheorem{theorem}{Theorem}[section]
\newtheorem{proposition}[theorem]{Proposition}
\newtheorem{lemma}[theorem]{Lemma}
\newtheorem{corollary}[theorem]{Corollary}

\theoremstyle{definition}
\newtheorem{definition}[theorem]{Definition}
\newtheorem{remark}[theorem]{Remark}

\newtheorem{assumption}[theorem]{Assumption}

\theoremstyle{remark}

\renewenvironment{thebibliography}[1]{%
\begin{oldthebibliography}{#1}%
\setlength{\baselineskip}{1em}
\linespread{.2}
\small
\setlength{\parskip}{0.25ex}%
\setlength{\itemsep}{.20em}%
}%
{%
\end{oldthebibliography}%
}
\newcommand{\eps}{\varepsilon}

\newcommand{\N}{\mathbb{N}}

\newcommand{\R}{\mathbb{R}}

\newcommand{\X}{\mathsf{X}}
\newcommand{\Y}{\mathsf{Y}}

\newcommand{\cB}{\mathcal{B}}
\newcommand{\cC}{\mathcal{C}}

\newcommand{\cP}{\mathcal{P}}

\DeclareMathOperator{\proj}{proj}

\DeclareMathOperator{\spt}{spt}

\newcommand{\bz}{\bar{z}}

\newcommand{\mykill}[1]{}
\numberwithin{equation}{section}

\usepackage[pdfborder={0 0 0}]{hyperref}
\hypersetup{
  urlcolor = black,
  pdfauthor = {Promit Ghosal, Marcel Nutz, Espen Bernton},
  pdfkeywords = {Entropic Optimal Transport; Schr\"odinger Bridge; Stability; Cyclical Invariance; Sinkhorn's Algorithm},
  pdftitle = {Stability of Entropic Optimal Transport and Schr\"odinger Bridges},
  pdfsubject = {Stability of Entropic Optimal Transport and Schr\"odinger Bridges},
  pdfpagemode = UseNone
}

\begin{document}

\title{\vspace{-1em}
 Stability of Entropic Optimal Transport\\and Schr\"odinger Bridges%
 \thanks{
 The authors are grateful to Julio Backhoff-Veraguas and Giovanni Conforti for insightful discussions that greatly contributed to this research.}
 }
\date{\today}
\author{
  Promit Ghosal%
  \thanks{Department of Mathematics, Massachusetts Institute of Technology, promit@mit.edu.
  }
  \and
  Marcel Nutz%
  \thanks{
  Departments of Statistics and Mathematics, Columbia University, mnutz@columbia.edu. Research supported by an Alfred P.\ Sloan Fellowship and NSF Grants DMS-2106056, DMS-1812661.}
  \and
  Espen Bernton%
  \thanks{
  Department of Statistics, Columbia University, eb3311@columbia.edu.
  }
  \and
  }
\maketitle \vspace{-1.2em}

\begin{abstract}
We establish the stability of solutions to the entropically regularized optimal transport problem with respect to the marginals and the cost function. The result is based on the geometric notion of cyclical invariance and inspired by the use of $c$-cyclical monotonicity in classical optimal transport. As a consequence of stability, we obtain the wellposedness of the solution in this geometric sense, even when all transports have infinite cost. More generally, our results apply to a class of static Schr\"odinger bridge problems including entropic optimal transport.
\end{abstract}

\vspace{.3em}

{\small

}
\vspace{.6em}

\section{Introduction}\label{se:intro}

Computational progress has lead to manifold applications of optimal transport in high-dimensional problems ranging from machine learning and statistics to image and language processing (e.g., \cite{WGAN.17,ChernozhukovEtAl.17,RubnerTomasiGuibas.00,AlvarezJaakkola.18}).
In this context, entropic regularization is crucial to enable efficient large-scale computation via Sinkhorn's algorithm, hence has become the focus of dozens of recent studies. We refer to \cite{CuturiPeyre.19} for a survey with numerous references.

Our main contribution is the stability of solutions to the entropically regularized optimal transport problem with respect to the marginals and the cost function. Parallel to the fundamental stability theorem in classical optimal transport, it justifies, for example, that approximations found by solving discretized problems indeed converge to the true solution when the cost function  is continuous.
Our results are stated in terms of cyclical invariance, a geometric notion inspired by the $c$-cyclical monotonicity property in classical optimal transport.
When the entropic transport problem has finite value, a coupling is cyclically invariant if and only if it is an optimal transport. Our stability theorem entails a general wellposedness result beyond the realm of  optimization: cyclical invariance singles out a unique coupling even if the transport problem has infinite value---i.e., all couplings have infinite cost---and therefore the paradigm of cost minimization does not differentiate couplings from one another. 

For ease of exposition, the Introduction focuses on entropic optimal transport. More general results are stated in Section~\ref{se:mainResults} using the language of Schr\"odinger bridges that turns out to be natural for our approach.
Given a measurable cost function $c: \X\times\Y\to\R_{+}$ on Polish probability spaces $(\X,\mu)$ and $(\Y,\nu)$, we consider the entropic optimal transport problem with regularization parameter $\eps>0$,
\begin{equation}\label{eq:EOT}
  \inf_{\pi\in\Pi(\mu,\nu)} \int_{\X\times\Y} c \,d\pi + \eps H(\pi|P), \quad P:=\mu\otimes\nu,
\end{equation}
where $\Pi(\mu,\nu)$ is the set of couplings and $H(\cdot|P)$ denotes relative entropy (or Kullback--Leibler divergence) with respect to the \textbf{p}roduct~$P$ of the marginals, defined as 
$H(\pi|P):=\int \log (\frac{d\pi}{dP}) \,d\pi$ for $\pi\ll P$ and $H(\pi|P):=\infty$ otherwise.
If the minimization~\eqref{eq:EOT} is finite; i.e., if
\begin{equation}\label{eq:finitenessCondEOT}
\mbox{there exists $\pi_{0}\in\Pi(\mu,\nu)$ with $\int c \,d\pi_{0}+H(\pi_{0}|P)<\infty$,}
\end{equation}
then it admits a unique minimizer $\pi\in\Pi(\mu,\nu)$ and moreover $\pi\sim P$.

\begin{definition}\label{de:cyclInvEOT}
  A coupling $\pi\in\Pi(\mu,\nu)$ is called \emph{$(c,\eps)$-cyclically invariant} if $\pi\sim P$ and its density admits a version $d\pi/dP: \X\times\Y\to(0,\infty)$ such that
\begin{equation*}%
  \prod_{i=1}^{k} \frac{d\pi}{dP}(x_{i},y_{i})=\exp\bigg(\!-\frac{1}{\eps}\bigg[\sum_{i=1}^{k}  c(x_{i},y_{i}) - \sum_{i=1}^{k}  c(x_{i},y_{i+1})\bigg]\bigg)\prod_{i=1}^{k} \frac{d\pi}{dP}(x_{i},y_{i+1})
\end{equation*}  
 for all $k\in\N$ and $(x_{i},y_{i})_{i=1}^{k} \subset \X\times\Y$, where $y_{k+1}:=y_{1}$.
\end{definition} 

By way of a factorization property that is equivalent to cyclical invariance, known results imply the following relation to the optimization~\eqref{eq:EOT}. 

\begin{proposition}\label{pr:cyclInvEOT}
  Let~\eqref{eq:EOT} be finite. Then $\pi\in\Pi(\mu,\nu)$ is the minimizer of~\eqref{eq:EOT} if and only if~$\pi$ is $(c,\eps)$-cyclically invariant.
\end{proposition}

See Section~\ref{se:proofOfMain} for details and references. We are mainly interested in optimal transport problems on Euclidean spaces~$\X,\Y$. However, the only particular property of such spaces that plays a role for our analysis is that Lebesgue's theorem on the differentiation of measures holds. Thus, we postulate that property (see Assumption~\ref{as:LebesgueDiff}) and otherwise allow for a general Polish setting. We can now state the aforementioned wellposedness result.

\begin{theorem}[Wellposedness]\label{th:wellposednessEOT}
  Let $c:\X\times\Y\to[0,\infty)$ be continuous, $\eps>0$ and $(\mu,\nu)\in\cP(\X)\times\cP(\Y)$. There exists a unique $(c,\eps)$-cyclically invariant coupling $\pi\in\Pi(\mu,\nu)$. If~\eqref{eq:EOT} is finite, $\pi$ is its unique minimizer.
\end{theorem}

Uniqueness follows from known facts and does not require the continuity of~$c$. On the other hand, existence beyond the framework of finite cost is a completely novel result. Rather than using convex analysis or variational arguments, it is based on the subsequent stability theorem for cyclical invariance. 
One example where wellposedness with infinite cost is of interest, is the statistical notion of rank recently proposed in~\cite{CuturiTeboulVert.19}.
Multivariate ranks have been defined in nonparametric statistics through Brenier's optimal transport map to extend the usual scalar notions and tests; see~\cite{ChernozhukovEtAl.17,DebSen.19,delBarrioEtAl.18,GhosalSen.19}. Leveraging the same idea but computationally less expensive, entropic optimal transport is used in~\cite{CuturiTeboulVert.19} to define ``differentiable ranks.'' Theorem~\ref{th:wellposednessEOT} allows one to naturally define such ranks for arbitrary distributions---like in the scalar case---without imposing a second moment condition.

\begin{theorem}[Stability]\label{th:stabilityEOT}
   For $n\geq 1$, let $(\mu_{n},\nu_{n})\in\cP(\X)\times\cP(\Y)$, let $\eps_{n}>0$ and let $c_{n}:\X\times\Y\to[0,\infty)$ be measurable. Let $\pi_{n}\in\Pi(\mu_{n},\nu_{n})$ be $(c_{n},\eps_{n})$-cyclically invariant. Suppose that $\mu_{n},\nu_{n}$ converge weakly to some limits $\mu,\nu$, that $\eps_{n}\to\eps>0$ and that $c_{n}$ converges uniformly on bounded sets to a continuous function~$c$. Then $\pi_{n}$ converges weakly to a limit $\pi\in\Pi(\mu,\nu)$ and $\pi$ is $(c,\eps)$-cyclically invariant.
\end{theorem}

If the involved optimization problems are finite, the theorem states the stability of the (entropic) optimal transport couplings. A simple yet important application is when the marginals $\mu,\nu$ are approximated by discrete measures, as it would be in a computational implementation. Even for this particular case, we are not aware of similar results in the literature.
We mention that continuity for the limiting cost~$c$ in Theorem~\ref{th:stabilityEOT} is a sharp condition: Proposition~\ref{pr:discont} will demonstrate that any nontrivial discontinuity in~$c$ leads to a failure of Theorem~\ref{th:stabilityEOT}, for suitable marginals.

A noteworthy application of the stability theorem was presented in the follow-up work~\cite{Nutz.20} where it was observed that convergence of Sinkhorn's algorithm can be seen as a stability problem. In that context, the marginals $\mu_{n},\nu_{n}$ are produced by the algorithm and known to converge to $\mu,\nu$ in great generality, while the iterates of the algorithm correspond to $\pi_{n}$. Theorem~\ref{th:stabilityEOT} then implies the weak convergence to the correct optimizer~$\pi$. Its geometric approach completely avoids the difficulty of establishing the integrability properties of $\mu_{n},\nu_{n}$ or even the finiteness of the associated entropic optimal transport problems.

The general existence result of Theorem~\ref{th:wellposednessEOT} is a consequence of Theorem~\ref{th:stabilityEOT} applied with $c_{n}=c$, $\eps_{n}=\eps$ and approximations $(\mu_{n},\nu_{n})\to(\mu,\nu)$ where $\mu_{n},\nu_{n}$ are discrete measures with finite support.
To solve the problem with marginals $(\mu_{n},\nu_{n})$, one could use Proposition~\ref{pr:cyclInvEOT}, but this particular case is a finite-dimensional minimization problem that can also be solved by standard calculus arguments. In particular, Theorem~\ref{th:stabilityEOT} yields an approach to construct cyclically invariant couplings which is novel even when the optimization problem is finite. This approach does not use the (classical but non-trivial) arguments of convex analysis and density factorization behind Proposition~\ref{pr:cyclInvEOT} (see \cite{BorweinLewis.92, BorweinLewisNussbaum.94, Csiszar.75, FollmerGantert.97, RuschendorfThomsen.93, RuschendorfThomsen.97}, among others). It is also quite different from the iterative method of~\cite{Fortet.40} which uses another finiteness condition; see~\cite{Leonard.19} for a modern presentation, analysis and extension of that method. Instead, our approach is close in spirit to the construction of $c$-cyclically monotone couplings that is standard in classical optimal transport; cf.\ \cite[pp.\,64--65]{Villani.09}.

The analogy with classical optimal transport extends in several directions. McCann showed in~\cite{McCann.95} that cyclical monotonicity singles out a particular transport map for quadratic cost~$c$ on $\R^{d}$ even if the optimal transport problem (here the 2-Wasserstein distance) is infinite, thus extending Brenier's map to this setting; see also \cite[pp.\,249--258]{Villani.09} for more general results. Here, the analogy becomes precise in the limit $\eps\to0$:  the extended Brenier coupling is the weak limit of the couplings $\pi=\pi_{\eps}$ established in Theorem~\ref{th:wellposednessEOT} (this follows from~\cite{BerntonGhosalNutz.21}).  
Another important parallel occurs at the technical level. Working with cyclical invariance, we proceed in a local fashion and focus on finitely many points $(x_{i},y_{i})$ at a time, rather than working with global objects like the Schr\"odinger potentials and their function spaces. For instance, non-compact marginal supports do not cause any particular difficulty in this approach. We emphasize that the novelty of the present study lies in how cyclical invariance is exploited and proved for the limit; the invariance property itself is merely an equivalent way of stating a factorization property of the density that is well known (see~\cite{BorweinLewis.92} or Proposition~\ref{pr:weakIsStrong} below). It turns out that, once the line of argument is found, remarkably general results can be obtained with fairly concise proofs.
We may hope that the techniques developed here can yield further insights into asymptotic questions on entropic optimal transport, and thus view the technique itself as a central contribution. Such questions may include quantifying the speed of convergence in our stability theorem or finding analogous results for dynamic Schr\"odinger bridge problems.

The companion paper \cite{BerntonGhosalNutz.21} illustrates the use of cyclical invariance for the limit $\eps\to0$. In this degenerate asymptotic, the limiting object is classical optimal transport as characterized by $c$-cyclical monotonicity. The latter property describes the shape of the support of the coupling and, therefore, is readily amenable to weak convergence arguments via Portmanteau's theorem. The same fact is often exploited in classical optimal transport theory, for instance in the standard proof of its stability theorem \cite[p.\,77]{Villani.09}. In the present study, the limit is entropic optimal transport. Being a property of the density, the relation of cyclical invariance with weak convergence is less direct (especially as the measures in Theorem~\ref{th:stabilityEOT} may well be mutually singular). Our general principle is to blow the points $(x_{i},y_{i})$ in Definition~\ref{de:cyclInvEOT} up to small balls, pass to the weak limit, and then recover information about the limiting density by shrinking the balls, via differentiation of measures. This technique appears to be novel in this area.

Starting with~\cite{Mikami.02, Mikami.04}, a number of works examine the degenerate asymptotic $\eps\to0$ where the limiting problem is classical optimal transport. The fact that weak limits of entropic optimizers are optimal transports was established by~\cite{Leonard.12} using Gamma-convergence arguments in a more general context of Schr\"odinger bridges; see also~\cite{CarlierDuvalPeyreSchmitzer.17} for the case of optimal transport with quadratic cost. 
As mentioned above, \cite{BerntonGhosalNutz.21} extends this result to transport problems with infinite value by way of cyclical invariance and $c$-cyclical monotonicity; moreover, a large deviations principle quantifies the local rate of convergence. 
Related results can be found in~\cite{ConfortiTamanini.19,Pal.19} where the expansion of the optimal cost as a function of~$\eps$ is studied. 
We remark that Gamma convergence seems difficult to use in the context of Theorem~\ref{th:stabilityEOT} due to the reference measures changing along the sequence.
The limit $\eps\to0$ can also be analyzed in the associated dual problem, here the solutions are called potentials. Convergence of potentials was shown in~\cite{GigliTamanini.21} for quadratic costs and compactly supported marginals, and recently in~\cite{NutzWiesel.21} for a general Polish setting.
Closer to the problem occurring in computational practice as well as the present question of stability, \cite{Berman.20} studies the convergence of the discrete Sinkhorn algorithm to an optimal transport in the joint limit when $\eps_{n}\to0$ and the marginals $\mu,\nu$ are approximated by discretizations $\mu_{n},\nu_{n}$ satisfying a certain density property. Explicit error bounds are derived, for instance for quadratic cost on the torus, to establish near-linear complexity of the resulting algorithm. For more on the computational challenges and remedies in this regime, see for instance \cite{Schmitzer.19} and the references therein.

While we are not aware of general stability results for the nondegenerate limit $\eps_{n}\to\eps>0$ in the literature, the sampling complexity of entropic optimal transport (with fixed $\eps$) can be seen as a particular form of stability with respect to the marginals. Indeed, \cite{GenevayPeyreCuturi.18,MenaWeed.19} study how the the empirical entropic Wasserstein distance, obtained by optimally coupling i.i.d.\ samples from the marginals, converges to the population version. The results are based on global arguments exploiting the regularity of the Schr\"odinger potentials, which, in turn, is achieved by imposing compactness and decay conditions on the marginals. See also~\cite{HarchaouiLiuPal.20} which studies a related asymptotic regime for a different regularization of optimal transport.
Related to the present work at least in spirit, there are several areas where analogues of $c$-cyclical monotonicity have recently lead to breakthroughs, including martingale optimal transport~\cite{BeiglbockJuillet.12}, optimal Skorokhod embeddings~\cite{BeiglbockCoxHuesmann.14, BeiglbockNutzStebegg.21} and weak transport~\cite{BackhoffBeiglbockPammer.19}.

The remainder of this paper is organized as follows. Section~\ref{se:mainResults} details the setting and main results in the language of Schr\"odinger bridges. The first step towards the stability theorem is reported Section~\ref{se:absCont} where we establish that weak limits of cyclically invariant couplings remain absolutely continuous. This is based on comparing measures of rectangles, an analysis that may be of independent interest. Section~\ref{se:InvarOfLimits} continuous the main proof by showing that limits of cyclically invariant couplings are again cyclically invariant. It comprises of two steps; the aforementioned principle of blowing up points and passing to the limit first yields a weakened version of the invariance property, and then measure-theoretic arguments can be used to show that the (proper) invariance property already follows. The concluding Section~\ref{se:proofOfMain} collects the arguments to prove the main results and their ramifications, including that continuity of the cost function is necessary for  stability.

\section{Main Results}\label{se:mainResults}

Let $(\X,d)$ be a complete, separable metric space; we write $\cP(\X)$ for the space of probability measures on the Borel $\sigma$-field $\cB(\X)$ endowed with weak convergence (induced by bounded continuous functions). The same is assumed for the second marginal space $(\Y,d)$, and we equip $\X\times\Y$ with the metric $d((x,y),(x',y'))=\max\{d(x,x'),d(y,y')\}$. Throughout this section, two measures $(\mu,\nu)\in\cP(\X)\times\cP(\Y)$ play the role of given marginals for the static Schr\"odinger bridge problem
\begin{equation}\label{eq:SB}
   \inf_{\pi\in\Pi(\mu,\nu)} H(\pi|R)
\end{equation}
where $R\in\cP(\X\times\Y)$ is a given \textbf{r}eference measure. We refer to \cite{Follmer.88, Leonard.14} for extensive surveys on Schr\"odinger bridges.
The entropic optimal transport problem~\eqref{eq:EOT} can be recovered (up to constants) as a special case for~$R$  defined by
\begin{equation}\label{eq:defR}
  \frac{dR}{dP} = a e^{-c/\eps}, \quad P:=\mu\otimes\nu 
\end{equation}
where $a = (\int e^{-c/\eps}\,dP)^{-1}$ is the normalizing constant. In particular, $R\sim P$, which will also be an important condition in many of our results for~\eqref{eq:SB}.
By way of~\eqref{eq:defR}, the following generalizes Definition~\ref{de:cyclInvEOT}.

\begin{definition}\label{de:cyclInv}
  Let $\pi\in\Pi(\mu,\nu)$ and $R\in\cP(\X\times\Y)$. We call $(\pi,R)$ \emph{cyclically invariant} if $\pi\sim R\sim P$ and there exists a version $d\pi/dR :\X\times\Y\to(0,\infty)$ of the relative density satisfying 
\begin{equation*}%
  \prod_{i=1}^{k} \frac{d\pi}{dR}(x_{i},y_{i})=\prod_{i=1}^{k} \frac{d\pi}{dR}(x_{i},y_{i+1})
\end{equation*}  
 for all $k\in\N$ and $(x_{i},y_{i})_{i=1}^{k} \subset \X\times\Y$, where $y_{k+1}:=y_{1}$.
\end{definition}

The analogue of the finiteness condition~\eqref{eq:finitenessCondEOT} is that
\begin{equation*}\label{eq:finitenessCond}
\mbox{there exists $\pi_{0}\in\Pi(\mu,\nu)$ with $H(\pi_{0}|R)<\infty$.}
\end{equation*}
We summarize the pertinent facts; see Lemma~\ref{le:uniquenessAndMinimizer} for detailed references.

\begin{proposition}\label{pr:cyclInv}
  Let $R\sim P$ and let~\eqref{eq:SB} be finite. There exists a unique minimizer $\pi\in\Pi(\mu,\nu)$ for~\eqref{eq:SB}, it satisfies~$\pi\sim R$, and it is the unique coupling $\pi$ such that $(\pi,R)$ is cyclically invariant.
\end{proposition}

In the remainder of the paper we assume that the underlying spaces allow for differentiation of measures in the following sense.

\begin{assumption}\label{as:LebesgueDiff}
Given $\rho,\lambda\in\cP(\X)$ satisfying $\rho\ll \lambda$, there exists $\X'\subset \X$ of full $\lambda$-measure such that 
\begin{align}
f(x):=\lim_{r\to 0} \frac{\rho(B_r(x))}{\lambda(B_r(x))}, \quad x\in \X'
\end{align}
defines a version of the Radon--Nikodym density $d\rho/d\lambda$. The analogous property is assumed on the space~$\X\times\Y$.
\end{assumption}

For Euclidean spaces $\X,\Y$, the assumption holds by the standard differentiation theorem~\cite[Theorem~1.32, p.\,53]{EvansGariepy.15}. More generally, it holds in the context of so-called Vitali covering relations; the classical reference is~\cite[Theorem~2.9.8, p.\,156]{Federer.69}. For example, Assumption~\ref{as:LebesgueDiff} holds when~$\X$ and~$\X\times\Y$ are compact subsets of Riemannian manifolds (due to the ``directionally limited'' property established in~\cite[Section~2.8.9, pp.\,145-146]{Federer.69}), or more generally, countable unions of such sets. See also~\cite[ pp.\,4-8, esp.\ Example~1.15]{Heinonen.01} for an accessible introduction. For our purposes, the main restriction is that differentiation of measures generally fails on infinite-dimensional spaces; 
see~\cite{Preiss.81} for a counterexample and~\cite{Preiss.79} for a related result on coverings. %
An alternative to Assumption~\ref{as:LebesgueDiff}, making our results slightly more general, is to impose a doubling condition on the specific marginal measures~$(\mu,\nu)$; cf.\ Remark~\ref{rk:doubling}.

We have seen in the Introduction that continuity of~$c$ is essential for the stability of entropic optimal transport. In view of~\eqref{eq:defR}, it is then clear that the regularity of~$dR/dP$ is pivotal. The following generalizations of Theorems~\ref{th:wellposednessEOT} and~\ref{th:stabilityEOT} are our main results.

\begin{theorem}[Wellposedness]\label{th:wellposedness}
  Suppose that $R\sim P:=\mu\otimes\nu$ and that the density $dR/dP$ admits a continuous version. Then there exists a unique coupling $\pi\in\Pi(\mu,\nu)$ such that $(\pi,R)$ is cyclically invariant. If~\eqref{eq:SB} is finite, $\pi$ is its minimizer.
\end{theorem}

\begin{theorem}[Stability]\label{th:stability}
  For $n\geq 1$, consider $(\mu_{n},\nu_{n})\in\cP(\X)\times\cP(\Y)$, $R_{n}\in\cP(\X\times\Y)$ and $\pi_{n}\in\Pi(\mu_{n},\nu_{n})$. Let $(\pi_{n},R_{n})$ be cyclically invariant and suppose that $\mu_{n},\nu_{n},R_{n}$ converge weakly to some limits $\mu,\nu,R$, where $R\sim P:=\mu\otimes\nu$. Writing $P_{n}:=\mu_{n}\otimes\nu_{n}$, suppose also that for some versions $\frac{dR_{n}}{dP_{n}},\frac{dR}{dP}: \X\times\Y\to (0,\infty)$ of the respective densities and some constants $\alpha_{n}>0$, it holds that for any fixed $z\in\spt R$,
\begin{equation}\label{eq:mainCond}
     \frac{dR_{n}}{dP_{n}}(z')=[1+o(1)]\alpha_{n} \frac{dR}{dP}(z),
  \end{equation}
  where $o(1)$ stands for a function $\phi_{z}(z',n)\to0$ as $d(z',z)+1/n\to0$.
Then~$\pi_{n}$ converges weakly to a limit $\pi\sim R$ and $(\pi,R)$ is cyclically invariant.
\end{theorem}

Schr\"odinger bridges are closely related to so-called Schr\"odinger systems (also called Schr\"odinger equations). For instance, Theorem~\ref{th:wellposedness} entails the following wellposedness result.

\begin{corollary}[Schr\"odinger System]\label{co:Schr\"odingerEqn}
  Let $(\mu,\nu)\in\cP(\X)\times\cP(\Y)$ and let $f:\X\times\Y\to(0,\infty)$ be continuous with $\int_{\X\times\Y} f\,dP=1$. There exist Borel functions $\varphi:\X\to(0,\infty)$ and $\psi:\Y\to(0,\infty)$ such that 
  \begin{equation}\label{eq:Schr\"odingerEqn}
    \int_{\Y} f(x,y)\psi(y)\,\nu(dy)=\varphi(x)^{-1}, \quad \int_{\X} f(x,y)\varphi(x)\,\mu(dx)=\psi(y)^{-1}
  \end{equation} 
  for $\mu$-a.e.\ $x\in \X$ and $\nu$-a.e.\ $y\in \Y$. The pair $(\varphi,\psi)$ is a.s.\ unique up to a multiplicative constant.\footnote{I.e., any solution $(\varphi',\psi')$ satisfies $\varphi'=a\varphi$ $\mu$-a.s., $\psi'=a^{-1}\psi$ $\nu$-a.s., for some $a>0$.}
\end{corollary} 

As above, the uniqueness follows from known results. Existence for continuous functions $f$ such that $f, f^{-1}$ are uniformly bounded was first shown in~\cite{Beurling.60}.
Under the boundedness condition alone, existence is due to~\cite{HobbyPyke.65}. Using the connection with Schr\"odinger bridges, \cite{RuschendorfThomsen.93} relaxed the boundedness to a condition of finite entropy, corresponding to the finiteness of~\eqref{eq:SB} in our setting. We refer to~\cite{Leonard.14} for a more complete review of the literature which dates back to Schr\"odinger. In Corollary~\ref{co:Schr\"odingerEqn}, we reintroduce the continuity condition of~\cite{Beurling.60} but avoid any condition of finite entropy or boundedness. Of course, the stability result of Theorem~\ref{th:stability} also has an analogous corollary for Schr\"odinger systems.

\section{Absolute Continuity of Limits}\label{se:absCont}

In this section we show that if $(\pi_{n},R_{n})$ is cyclically invariant and $R_{n}\sim \mu_{n}\otimes\nu_{n}$ holds in a locally uniform sense (to be made precise), then any weak limit pair $(\pi,R)=(\lim_{n} \pi_{n},\lim_{n} R_{n})$ must satisfy $\pi\ll R$.

As the method of proof is novel, we first sketch the line of argument.
We shall be comparing the measures of rectangles $F_{i}\times G_{i}\subset \X\times\Y$ for $i=1,2$ with their permutations $F_{i}\times G_{i+1}$.\footnote{We use the cyclical convention for $i\in\{1,2\}$; that is, $i+1:=1$ for $i=2$.}
Consider first the trivial case $R_{n}=P_{n}:=\mu_{n}\otimes\nu_{n}$, then cyclical invariance of $(\pi_{n},R_{n})$ implies
\begin{equation}\label{eq:absContIdea}
  \pi_{n}(F_{1}\times G_{1})\pi_{n}(F_{2}\times G_{2}) = \pi_{n}(F_{1}\times G_{2})\pi_{n}(F_{2}\times G_{1}).
\end{equation}
This will of course no longer hold if $R_{n}\neq P_{n}$, but the equivalence $R_{n}\sim P_{n}$ suggests that the two sides of~\eqref{eq:absContIdea} should still be comparable. We will quantify the equivalence $R_{n}\sim P_{n}$ and assume it to hold uniformly in~$n$. Then, we prove that the two sides of~\eqref{eq:absContIdea} are comparable in the sense that their quotient remains bounded uniformly in~$n$. For suitable rectangles, the bound propagates to the weak limit $(\pi,R)$, accomplishing the first step of the proof.
The second step is to argue by contraposition that this bound implies $\pi\ll R$. Indeed, we establish that if $\pi\in\Pi(\mu,\nu)$ is singular wrt.\ $\mu\otimes\nu$, then there exist $F_{i},G_{i}$ such that $\pi(F_{1}\times G_{1})\pi(F_{2}\times G_{2})$ is above a threshold whereas $\pi(F_{1}\times G_{2})\pi(F_{2}\times G_{1})$ is arbitrarily small.

For ease of reference, we first record two measure-theoretic facts. %

\begin{lemma}\label{le:continuitySets}
  Let $\rho$ be a $\sigma$-finite measure on a Polish space~$(\Omega,d)$. We say that $C\subset\Omega$ is $\rho$-\emph{continuous} if its boundary $\partial C$ is a $\rho$-nullset.
  
  (i) $\rho$-continuous sets form a field; i.e., unions, intersections, complements, differences of $\rho$-continuous sets are again $\rho$-continuous.
  
  (ii) For fixed $z\in \Omega$, the open ball $B_{r}(z)=\{z': d(z,z')<r\}$ is $\rho$-continuous for all but countably many values of~$r>0$. In particular, given $r>0$, there exists $0<r'\leq r$ such that $B_{r'}(z)$ is $\rho$-continuous.
  
  (iii) If $F\subset\X$ is $\mu$-continuous and $G\subset\Y$ is $\nu$-continuous, then $F\times G$ is $\pi$-continuous for any $\pi\in\Pi(\mu,\nu)$.

  (iv) Given $A\in\cB(\Omega)$ and $\eps>0$, there exists an open set $B\subset A_{\eps}:=\{d(\cdot,A)<\eps\}$ with $\rho(\partial B)=0$ and $\rho(A\Delta B)<\eps$.
\end{lemma}

\begin{proof}
 Statement (i) is verified directly; (ii) holds because
 a $\sigma$-finite measure admits at most countably many disjoint sets of positive measure; (iii) follows from $\partial (F\times G)=(\overline F \times \partial G) \cup (\partial F \times \overline G)$. Let $A,\eps$ be as in~(iv). By interior and exterior regularity of $\rho$, there is a compact set $K\subset A$ with $\rho(A\setminus K)<\eps$ and an open set $A\subset O \subset A_{\eps}$ with $\rho(O\setminus A)<\eps$. 
  We have $r(z):=d(z,O^{c})>0$ for all $z\in K$ by the closedness of $O^{c}$, and~$K$ is covered by the open balls $B_{r(z)}(z)$ with $z\in K$. 
  After making $r(z)$ smaller if necessary, each of these balls is $\rho$-continuous. Choosing a finite cover $\{B_{r(z_{i})}(z_{i})\}_{i\leq N}$, the set $B:=\cup_{i} B_{r(z_{i})}(z_{i})$ has the required properties.
\end{proof} 

The second fact is a conditional version of the differentiation of measures, based on Assumption~\ref{as:LebesgueDiff} for the marginal space~$\X$. While not widely known, this concept was already established in~\cite{Pfanzagl.79}, although the author defined differentiation of measures in a slightly different way. For the convenience of the reader, we detail the adaptation to our setting.

\begin{lemma}\label{le:conditionalDiffOfMeas}
Let $\pi\in\cP(\X\times\Y)$ and let $\mu$ be its first marginal. Consider for $x\in\spt \mu$ and $r>0$ the probability measure $\pi^{(r)}_{x}\in\cP(\Y)$ defined by
$$
  \pi^{(r)}_{x}(C) = \frac{\pi(B_r(x)\times C)}{\mu(B_r(x))}, \quad C \in \mathcal{B}(\Y).
$$
Under Assumption~\ref{as:LebesgueDiff} on~$\X$, there exists $\X_{0}\subset\spt\mu$ with $\mu(\X_0)= 1$ such that for all $x\in \X_0$, the weak limit
$$
 \pi_x := \lim_{r\to0} \pi^{(r)}_x
$$
exists. Moreover, $\pi_{x}$ defines a regular conditional probability of~$\pi$ given $x$.
\end{lemma}

\begin{proof}
As $\X,\Y$ are Polish, there exists some regular conditional probability~$\hat{\pi}_{x}$;
it suffices to show that $\lim_{r\to0} \pi^{(r)}_x=\hat{\pi}_{x}$ weakly for $\mu$-a.e.\ $x\in\X$. 
Fix a countable collection $\cC$ of nonnegative bounded continuous test functions~$\phi:\Y\to\R$ that determine weak convergence (cf.\ \cite[Theorem~6.6, p.\,47]{Parthasarathy.67}).
Then we need to show, for fixed $\phi\in\cC$, that 
\begin{align}\label{eq:RCP}
\int \phi\,d\pi^{(r)}_x\to \int \phi\,d\hat{\pi}_{x} \quad\mbox{for $\mu$-a.e.\ $x\in\X$. }
\end{align}
As $\hat{\pi}_{x}$ is a regular conditional probability, it holds for $\mu$-a.e.\ $x\in\X$ that 
\begin{align*}%
\int \phi\,d\pi^{(r)}_x &= \frac{1}{\mu(B_r(x))} \int_{B_r(x)} \mu(dx')\int \phi(y) \,\hat{\pi}_{x'}(dy) \\
&= \frac{1}{\mu(B_r(x))} \int_{B_r(x)} f \, d\mu, \qquad f(x'):=\int \phi(y) \,\hat{\pi}_{x'}(dy).
\end{align*}
We now apply Assumption~\ref{as:LebesgueDiff} to the pair $f \, d\mu \ll d\mu$ and deduce that the right-hand side converges to $f(x)$ for $\mu$-a.e.\ $x\in\X$, which is~\eqref{eq:RCP}.
\end{proof}

The next result is the main ingredient for the second step as sketched above: the sets $\{x_{i}\}\times U_{i}$ constructed in Lemma~\ref{le:forAcContradict} will be ``blown up'' to rectangles in the proof of Proposition~\ref{pr:limitAbsCont} and used to show by contraposition that $\pi\ll R$.

\begin{lemma}\label{le:forAcContradict}
  Let $\pi\in\Pi(\mu,\nu)$ and let $\pi=\mu(dx)\otimes \pi_{x}(dy)$ be a disintegration. If $\pi\not\ll \mu\otimes\nu$, then  $$
    \mu^{2}\{(x_{1},x_{2})\in\X^{2}:\, \exists  U_{1},U_{2}\in\cB(\Y) \mbox{ with } \pi_{x_{i}}(U_{ i})>0, \pi_{x_{i}}(U_{i+1})=0\}>0.
  $$
  In addition, the sets $U_{i}$ can be chosen to be disjoint, of arbitrarily small diameter, and such that $\pi_{x_{i}}(\partial U_{j})=\nu(\partial U_{j})=0$ for $i,j\in\{1,2\}$.
\end{lemma} 

\begin{proof}
  Let $\pi\not\ll \mu\otimes\nu$; that is, there exists a set $A\in\cB(\X\times\Y)$ with $\pi(A)>0$ and $(\mu\otimes\nu)(A)=0$. Let $A_{x}=\{y:\, (x,y)\in A\}$ denote the $x$-section, then $\nu(A_{x})=0$ for $\mu$-a.e.\ $x\in\X$. On the other hand, any $B\in\cB(\Y)$ satisfies $\nu(B)=\int_{\X} \pi_{x'}(B)\,\mu(dx')$. In particular, $\nu(A_{x})=0$ implies that $\pi_{x'}(A_{x})=0$ for $\mu$-a.e.\ $x'\in\X$. Therefore, 
  $$
    F=\{(x,x')\in\X^{2}:\, \pi_{x'}(A_{x})=0\} \quad \mbox{has full measure } \mu^{2}.
  $$
  As $\pi(A)>0$, the set $W=\{x:\,\pi_{x}(A_{x})>0\}$ satisfies $\mu(W)>0$ and hence $\mu^{2}((W\times W) \cap F)>0$. For $(x_{1},x_{2})\in (W\times W) \cap F$ we have $\pi_{x_{i}}(A_{x_{i}})>0$ and $\pi_{x_{i}}(A_{x_{i+1}})=0$. In particular, the disjoint sets $U'_{i}:=A_{x_{i}}\setminus A_{x_{i+1}}$ satisfy $\pi_{x_{i}}(U'_{i})>0$ and $\pi_{x_{i}}(U'_{i+1})=0$. By intersecting with a suitable ball, the diameter of $U'_{i}$ can be assumed to be arbitrarily small. Finally, let $\rho=\pi_{x_{1}}+\pi_{x_{2}}+\nu$ and choose $\rho$-continuous sets $U''_{i}$ for $U'_{i}$ as in Lemma~\ref{le:continuitySets}\,(iv), with $\eps>0$ small enough such that the sets $U_{i}:=U''_{i}\setminus U''_{i+1}$ have the required properties; cf.\ Lemma~\ref{le:continuitySets}\,(i).
\end{proof}

The next lemma establishes that the two sides of~\eqref{eq:absContIdea} are comparable with a bound related to the equivalence $R_{n}\sim P_{n}$. The following notation is useful: when $k\geq 1$ and a $k$-tuple $(z_{1},\dots, z_{k})\in(\X\times\Y)^{k}$ are given, and 
\begin{equation}\label{eq:barDefn}
 \mbox{if} \quad z_{i}=(x_{i},y_{i})\in \X\times\Y, \quad\mbox{we set} \quad \bz_{i}:=(x_{i},y_{i+1}),
\end{equation}
with the cyclical convention~$y_{k+1}:=y_{1}$.

\begin{lemma}\label{le:monotoneIntegrated}
  Let $\pi\in\Pi(\mu,\nu)$ and $\pi\ll R\sim P:=\mu\otimes\nu$. Consider rectangles $A_{i}\in\cB(\X)\times\cB(\Y)$ for $1\leq i\leq k$ and denote $\bar{A}_{1}\times \dots \times \bar{A}_{k}:=\{(\bz_{1},\dots,\bz_{k}): z_{i}\in A_{i}\}$.
 For some $\alpha,\bar{\alpha}>0$, suppose that $dR/dP\leq \alpha$ on $A_{i}$ and $(dR/dP)^{-1}\leq \bar{\alpha}$ on $\bar{A}_{i}$, for all~$i$. Then
  $$
    \pi^{k}(A_{1}\times \dots \times A_{k}) \leq (\alpha\bar{\alpha})^{k} \pi^{k}(\bar{A}_{1}\times \dots \times \bar{A}_{k}).
  $$ 
\end{lemma} 

\begin{proof}
  Set $Z=d\pi/dR$ and note that
  \begin{equation}\label{eq:Pinvariant}
    P(dz_{1})\cdots P(dz_{k}) = \mu(dx_{1})\cdots \mu(dx_{k})\nu(dy_{1})\cdots \nu(dy_{k})=P(d\bz_{1})\cdots P(d\bz_{k}).
  \end{equation}
  Using the cyclical invariance of $(\pi,R)$ and the rectangular form of $A_{i}$,
  \begin{align*}
    &\pi^{k}(A_{1}\times \dots \times A_{k}) \\
    & = \int_{A_{1}\times \dots \times A_{k}} Z(z_{1})\cdots Z(z_{k})\,\frac{dR}{dP}(z_{1})\cdots \frac{dR}{dP}(z_{k})\, P(dz_{1})\cdots P(dz_{k}) \\
    & \leq \alpha^{k}\int_{A_{1}\times \dots \times A_{k}} Z(z_{1})\cdots Z(z_{k})\, P(dz_{1})\cdots P(dz_{k}) \\
    & = \alpha^{k}\int_{A_{1}\times \dots \times A_{k}} Z(\bz_{1})\cdots Z(\bz_{k})\, P(d\bz_{1})\cdots P(d\bz_{k}) \\
    & = \alpha^{k}\int_{\bar{A}_{1}\times \dots \times \bar{A}_{k}} Z(z_{1})\cdots Z(z_{k})\, P(dz_{1})\cdots P(dz_{k}) \\
    & \leq (\alpha\bar{\alpha})^{k} \int_{\bar{A}_{1}\times \dots \times \bar{A}_{k}} Z(z_{1})\cdots Z(z_{k})\,  R(dz_{1})\cdots R(dz_{k}) \\
    & = (\alpha\bar{\alpha})^{k} \pi^{k}(\bar{A}_{1}\times \dots \times \bar{A}_{k}).\qedhere
  \end{align*}
\end{proof}

We can now prove the main result of this section.

\begin{proposition}\label{pr:limitAbsCont}
  Let $(\mu_{n},\nu_{n})\in\cP(\X)\times\cP(\Y)$, let $\pi_{n}\in\Pi(\mu_{n},\nu_{n})$ and $R_{n}\sim P_{n}:=\mu_{n}\otimes\nu_{n}$. Suppose that $(\pi_{n},R_{n})$ is cyclically invariant for each~$n$ and that $\pi_{n}$ converges weakly to some limit $\pi$. In particular, $\mu_{n},\nu_{n}$ converge to some limits $\mu,\nu$, and $\pi\in\Pi(\mu,\nu)$. Set $P= \mu\otimes\nu$ and suppose that 
$(R_{n},P_{n})$ are uniformly locally equivalent in the following sense: there are versions $\frac{dR_{n}}{dP_{n}}: \X\times\Y\to(0,\infty)$ of the relative densities and constants $\alpha_{n}>0$ such that 
given $z\in\spt P$, there exist $r=r(z)>0$ and $n_{0}=n_{0}(z,r)$ with 
  \begin{equation}\label{eq:condForlimitAbsCont}
   \sup_{z'\in B_{r}(z),\,n\geq n_{0}} \bigg(\alpha_{n}\frac{dR_{n}}{dP_{n}}(z') + \alpha_{n}^{-1}\frac{dP_{n}}{dR_{n}}(z')\bigg)<\infty.
  \end{equation}
  Then $\pi\ll P$.
\end{proposition} 

\begin{proof}
  Note that the weak convergence of $\pi_{n}\in\Pi(\mu_{n},\nu_{n})$ implies the weak convergence of its marginals to some limits $\mu$ and $\nu$, and then $\pi\in\Pi(\mu,\nu)$. By Lemma~\ref{le:conditionalDiffOfMeas}, there is a set $\X_{0}\subset\spt\mu$ of full $\mu$-measure such that for $x\in\X_{0}$, the weak limit
  $$
    \pi_{x}(\cdot):=\lim_{r\to0} \frac{\pi(B_{r}(x)\times \cdot)}{\mu(B_{r}(x))}
  $$
  exists, and $\pi=\mu\otimes_{x} \pi_{x}$ is a disintegration of~$\pi$.  In particular, for any $x\in\X_{0}$ and any $\pi_{x}$-continuity set $U$,
  \begin{equation}\label{eq:limitRCPD}
    \pi_{x}(U) = \lim_{r\to0} \frac{\pi(B_{r}(x)\times U)}{\mu(B_{r}(x))}.
  \end{equation}
  Suppose for contradiction that $\pi\not\ll \mu\otimes\nu$. Then Lemma~\ref{le:forAcContradict} yields $x_{i}\in\X_{0}$ ($i=1,2$) and disjoint sets $U_{i}$ such that 
  $$
  q:=\min_{i} \pi_{i}(U_{i})>0\quad \mbox{and}\quad \pi_{i}(U_{i+1})=0, \quad\mbox{where }\pi_{i}:=\pi_{x_{i}}.
  $$
  Moreover, given $r_{0}>0$, the sets $U_{i}$ can be chosen to be $(\pi_{1}+\pi_{2}+\nu)$-continuous and contained in a ball $B_{r_{0}}(y_{i})$ around some $y_{i}\in \spt\nu$. Using also~\eqref{eq:condForlimitAbsCont}, we can choose $r_{0},n_{0}$ such that for some $\alpha_{*}>0$, 
  \begin{equation}\label{eq:limitAbsContConsts}
    \alpha_{n}\frac{dR_{n}}{dP_{n}}\leq \alpha_{*}\quad \mbox{on } B_{r_{0}}(x_{i})\times U_{i},
    \qquad \alpha_{n}^{-1}\frac{dP_{n}}{dR_{n}}\leq \alpha_{*} \quad \mbox{on } B_{r_{0}}(x_{i})\times U_{i+1},
  \end{equation}
  for all $1\leq i \leq k$ and $n\geq n_{0}$.
  Writing $B_{i,r}:=B_{r}(x_{i})$ for brevity, \eqref{eq:limitRCPD} implies in particular that 
  \begin{equation*}
    \frac{\pi(B_{i,r}\times U_{i})}{\mu(B_{i,r})}\to  \pi_{i}(U_{i})\geq q, \quad \frac{\pi(B_{i,r}\times U_{i+1})}{\mu(B_{i,r})} \to \pi_{i}(U_{i+1}) =0.
  \end{equation*}  
  That is, given $\eps>0$, choosing $r\leq r_{0}$ small enough results in 
  \begin{equation}\label{eq:limitAbsContProof}
    \frac{\pi(B_{i,r}\times U_{i})}{\mu(B_{i,r})} \geq q-\eps, \quad \frac{\pi(B_{i,r}\times U_{i+1})}{\mu(B_{i,r})} \leq \eps,
  \end{equation}
  and we may further choose $r$ such that $\mu(\partial B_{i,r})=0$. %
   Specifically, we choose $\eps>0$ such that $\eps/(q-\eps)<\alpha_{*}^{-2}$, then~\eqref{eq:limitAbsContProof} implies
  $$
    \prod_{i=1,2} \pi(B_{i,r}\times U_{i}) > \alpha_{*}^{4} \prod_{i=1,2} \pi(B_{i,r}\times U_{i+1}).
  $$  
  Note that $B_{i,r}\times U_{j}$ is a $\pi$-continuity set for $i,j\in\{1,2\}$; cf. Lemma~\ref{le:continuitySets}\,(iii). In view of $\pi_{n}\to \pi$, it then follows that
  \begin{equation}\label{eq:limitAbsContProofContrad}
    \prod_{i=1,2} \pi_{n}(B_{i,r}\times U_{i}) > \alpha_{*}^{4} \prod_{i=1,2} \pi_{n}(B_{i,r}\times U_{i+1})
  \end{equation} 
  for $n$ sufficiently large. %
  On the other hand, we apply Lemma~\ref{le:monotoneIntegrated} with $k=2$ and $A_{i}=B_{i,r}\times U_{i}$. In view of~\eqref{eq:limitAbsContConsts}, the condition of the lemma holds with $\alpha:=\alpha_{n}^{-1}\alpha_{*}$ and $\bar\alpha:=\alpha_{n}\alpha_{*}$. Noting that the $\alpha_{n}$ cancel to yield $(\alpha\bar\alpha)^{k}=\alpha_{*}^{4}$, the lemma yields the inequality opposite to~\eqref{eq:limitAbsContProofContrad}, a contradiction.
\end{proof}

\section{Cyclical Invariance of Limits}\label{se:InvarOfLimits}

In this section we aim to show that limits of cyclically invariant couplings are again cyclically invariant, under suitable conditions. We proceed in two steps. First, we establish that limits are weakly cyclically invariant as defined below. Second, we show that weak cyclical invariance already implies cyclical invariance. This second step has little to do with the passage to the limit; rather, it settles some measure-theoretic aspects to get rid of pesky nullsets.

The ``weak'' notion is introduced mainly to disentangle the proof of the main result. It weakens in two ways the cyclical invariance of $(\pi,R)$ as stated in Definition~\ref{de:cyclInv}: the equivalence of~$\pi$ and $R$ is reduced to absolute continuity and the  cyclical relation only holds for points from specific sets. We recall the notation~$\bz_{i}$ from~\eqref{eq:barDefn}.

\begin{definition}\label{de:weaklyCyclInv}
  Let $\pi\in\Pi(\mu,\nu)$ and $R\in\cP(\X\times\Y)$. We call $(\pi,R)$ \emph{weakly cyclically invariant} if $\pi\ll R\sim P:=\mu\otimes\nu$ and there exist
  \begin{enumerate}
  \item a version $Z:\X\times\Y\to[0,\infty]$ of the density $d\pi/dR$,
  \item $\Omega_{1},\Omega_{0}\in\cB(\X\times\Y)$ with $\pi(\Omega_{1})=R(\Omega_{0})=1$ and $0<Z<\infty$ on $\Omega_{1}$
  \end{enumerate} 
 such that for all $k\in\N$,
 \begin{equation*}
  \prod_{i=1}^{k} Z(z_{i})= \prod_{i=1}^{k} Z(\bz_{i}) \quad \mbox{for all $(z_{i})_{i=1}^{k} \subset \Omega_{1}$ with $(\bz_{i})_{i=1}^{k} \subset \Omega_{0}$.}
\end{equation*}  
\end{definition}

\begin{proposition}\label{pr:limitEquiv}
  Let $(\mu_{n},\nu_{n})\in\cP(\X)\times\cP(\Y)$, let $\pi_{n}\in\Pi(\mu_{n},\nu_{n})$ and $R_{n}\sim P_{n}:=\mu_{n}\otimes\nu_{n}$. Suppose that $(\pi_{n},R_{n})$ is cyclically invariant for each~$n$ and that $\pi_{n},R_{n}$ converge weakly to some limits $\pi,R$. In particular, $\mu_{n},\nu_{n}$ converge to some limits $\mu,\nu$, and $\pi\in\Pi(\mu,\nu)$. Suppose that $R\sim P:=\mu\otimes\nu$ and that there are versions $f_{n},f: \Omega\to (0,\infty)$ of the densities $\frac{dR_{n}}{dP_{n}},\frac{dR}{dP}$ and constants $\alpha_{n}>0$ such that for any fixed $z\in\spt R$,
  \begin{equation}\label{eq:densityConvCond}
    f_{n}(z')=[1+o(1)]\alpha_{n} f(z),
  \end{equation}
  where $o(1)$ stands for a function $\phi_{z}(z',n)\to0$ as $d(z',z)+1/n\to0$.
  Then $(\pi,R)$ is weakly cyclically invariant. Specifically, the quantities $\Omega_{1},\Omega_{0},Z$ of Definition~\ref{de:weaklyCyclInv} can be chosen as
  \begin{align*}
    \Omega_{0} &= \left\{z\in\spt R:\, Z(z):=  \lim_{r\to 0} \frac{\pi(B_{r}(z))}{R(B_{r}(z))} \mbox{ exists in }[0,\infty)\right\}\\
    & \quad \cap \left\{z\in\spt P:\, f(z)=\lim_{r\to 0} \frac{R(B_{r}(z))}{P(B_{r}(z))}\right\}
  \end{align*}
  and $\Omega_{1}=\Omega_{0}\cap\{Z>0\}$.
\end{proposition}

\begin{proof}
  Assumption~\ref{as:LebesgueDiff} for the space $\X\times\Y$ shows that $Z$ is a version of $d\pi/dR$ and $R(\Omega_{0})=1$. As~\eqref{eq:densityConvCond} implies~\eqref{eq:condForlimitAbsCont}, Proposition~\ref{pr:limitAbsCont} yields that $\pi\ll P$, hence $\pi\ll R$ and the definition of~$\Omega_{1}$ implies $\pi(\Omega_{1})=1$.

  Let $z_{1},\dots, z_{k}\in\Omega_{1}$ be such that $\bz_{i}\in\Omega_{0}$ and consider for $r>0$ the balls $A_{i}=A_{i}^{(r)}:=B_{r}(z_{i})=B_{r}(x_{i})\times B_{r}(y_{i})$. To avoid unwieldy formulas, we use the vector notation
  \begin{align*}
    \bm{z}=(z_{1},\dots, z_{k}), \quad\! \bm{\bar{z}}=(\bz_{1},\dots, \bz_{k}),
    \quad\! 
    \bm{A} = A_{1}\times\cdots \times A_{k}, \quad\! \bm{\bar{A}} =\{\bm{\bar{z}}: \bm{z}\in \bm{A}\}
  \end{align*} 
  together with the convention that functions and measures are evaluated by multiplication over the components, for instance
  \begin{align*}
    f(\bm{z})=f(z_{1})\cdots f(z_{k}), \quad \pi(\bm{A}) = \pi(A_{1})\cdots \pi(A_{k}).
  \end{align*} 
  As $d\pi_{n}/dP_{n}=(d\pi_{n}/dR_{n}) (dR_{n}/dP_{n})$, we then have 
  $$
    \pi_{n}(\bm{A}) 
     = \int_{\bm{A}} Z_{n}(\bm{z}')\,f_{n}(\bm{z}')\, P_{n}(d\bm{z}'), \quad Z_{n}:=d\pi_{n}/dR_{n}.
   $$
   We also have $\pi_{n}(A_{i})>0$ for~$n$ large as $Z(z_{i})>0$ and $\pi_{n}\to \pi$.
   On the other hand, $Z_{n}(\bm{z}')$ and $P_{n}(d\bm{z}')$  are invariant under~$\bm{z}'\mapsto \bm{\bar{z}}'$ due to the assumption on $\pi_{n}$ and the form of $P_{n}$; cf.~\eqref{eq:Pinvariant}. Thus
  \begin{align*}
    \pi_{n}(\bm{\bar{A}}) 
     = \int_{\bm{\bar{A}}} Z_{n}(\bm{z}')\,f_{n}(\bm{z}')\, P_{n}(d\bm{z}')
     = \int_{\bm{A}} Z_{n}(\bm{z}')\,f_{n}(\bm{\bar{z}}')\, P_{n}(d\bm{z}').
   \end{align*}
   Applying the assumption on $f_{n}$ to each of the points $z_{i},\bz_{i}$ then yields
   \begin{align*} 
    \frac{\pi_{n}(\bm{\bar{A}})}{\pi_{n}(\bm{A})}
   &= \frac{\int_{\bm{A}} Z_{n}(\bm{z}')\,f_{n}(\bm{\bar{z}}')\, P_{n}(d\bm{z}')}
    {\int_{\bm{A}} Z_{n}(\bm{z}')\,f_{n}(\bm{z}')\, P_{n}(d\bm{z}')}\\
   &= \frac{[1+o(1)]\alpha_{n} f(\bm{\bar{z}}) \int_{\bm{A}} Z_{n}(\bm{z}')\, P_{n}(d\bm{z}')}
    {[1+o(1)]\alpha_{n} f(\bm{z}) \int_{\bm{A}} Z_{n}(\bm{z}')\, P_{n}(d\bm{z}')}
    = [1+o(1)]\frac{f(\bm{\bar{z}})}{f(\bm{z})}
  \end{align*}
  where $o(1)$ stands for a function of $(\bm{z},\bm{\bar{z}})$ converging to zero as $r+1/n\to0$.
   For values of $r>0$ such that $A_{i},\bar{A}_{i}$ are continuity sets of~$P$ (and hence also of~$\pi$ and~$R$), taking $n\to\infty$ yields
   \begin{align}\label{eq:proofLimitInvar1} 
    \frac{\pi(\bm{\bar{A}})}{\pi(\bm{A})}
   & = [1+o(1)]\frac{f(\bm{\bar{z}})}{f(\bm{z})}.
  \end{align}   
   On the other hand, $z_{i},\bz_{i}\in\Omega_{0}$ also guarantees that
  $
   \frac{R(A_{i})}{P(A_{i})} =  [1+o(1)] f({z}_{i})
  $ 
  and similarly for~$\bz_{i}$. Recalling  $P(\bm{\bar{A}})=P(\bm{A})$, we deduce
  \begin{align}\label{eq:proofLimitInvar2}
    \frac{R(\bm{\bar{A}})}{R(\bm{A})}
   & = \frac{R(\bm{\bar{A}})/P(\bm{\bar{A}})}{R(\bm{A})/P(\bm{A})}
   = [1+o(1)]\frac{f(\bm{\bar{z}})}{f(\bm{z})}.
  \end{align} 
  Combining~\eqref{eq:proofLimitInvar1} and~\eqref{eq:proofLimitInvar2} yields
  $$
    \frac{\pi(\bm{A})}{R(\bm{A})} = [1 + o(1)] \frac{\pi(\bm{\bar{A}})}{R(\bm{\bar{A}})}
  $$  
  and then letting $r\to 0$ (along a sequence of $r$ such that $A_{i},\bar{A}_{i}$ are continuity sets of~$P$) shows $Z(\bm{z})=Z(\bm{\bz})$, as desired.
\end{proof} 

\begin{remark}\label{rk: }
  Assumption~\eqref{eq:densityConvCond} on $f_{n},f$ in Proposition~\ref{pr:limitEquiv} is a sufficient condition for %
  \begin{align}\label{eq:limitEquivCrucial}
   \lim_{r\to0}\lim_{n\to\infty} \frac{\int_{\bm{A}} Z_{n}(\bm{z}')\,f_{n}(\bm{\bar{z}}')\, P_{n}(d\bm{z}')}
    {\int_{\bm{A}} Z_{n}(\bm{z}')\,f_{n}(\bm{z}')\, P_{n}(d\bm{z}')}
   =\frac{f(\bm{\bar{z}})}{f(\bm{z})};
  \end{align}
  it can be replaced by any other condition implying~\eqref{eq:limitEquivCrucial}. We note that~\eqref{eq:limitEquivCrucial} can be seen as a differentiation of measures intertwined with a weak limit. 
\end{remark} 

As mentioned above, the second step is to upgrade the weak cyclical invariance. Some of these considerations are similar to arguments in the proofs of \cite{BackhoffBeiglbockConforti.21}, where it is shown by variational arguments that minimizers of certain static Schr\"odinger bridge problems admit a factorization. %

\pagebreak[2]

\begin{proposition}\label{pr:weakIsStrong}
  Let $\pi\in\Pi(\mu,\nu)$ and let $R\in\cP(\X\times\Y)$ satisfy $R\sim P:=\mu\otimes\nu$. The following are equivalent:
  \begin{enumerate}
  \item $(\pi,R)$ is cyclically invariant,    
  \item $(\pi,R)$ is weakly cyclically invariant,
  \item there exist Borel functions $\varphi:\X\to(0,\infty)$ and $\psi:\Y\to(0,\infty)$ such that $(x,y)\mapsto \varphi(x)\psi(y)$ is a version of the density $d\pi/dR$.
  \end{enumerate}
\end{proposition}

The implications $(i)\Rightarrow(ii)$ and $(iii)\Rightarrow(i)$ are immediate. The fact that  $(i)\Rightarrow(iii)$ is also well known, cf.\ \cite{BorweinLewis.92}, but will not be used directly. For the proof of $(ii)\Rightarrow(iii)$, a measure-theoretic fact will be useful. Given a set~$A\subset\X\times\Y$ with $P(A)=1$, Lemma~\ref{le:connectingPoint} below states that within any set~$B$ of positive measure we can find a point~$(x_{*},y_{*})$ which acts like an airline hub for~$A$: any two points of~$A$ are connected through $(x_{*},y_{*})$, modulo marginal nullsets. In particular, any trip can be achieved with at most one stopover, and we may stop within~$B$. (The bound of one is optimal as the set $A$ need not be a rectangle; in fact, $A$ may fail to contain any measurable rectangle of positive measure \cite[Exercise~5.4, p.\,74]{Falconer.86}.) Lemma~\ref{le:connectingPoint} is refinement of \cite[Lemma~4.3]{BeiglbockGoldsternMareschSchachermayer.09} which asserts the connectedness of $A$ in the sense of~\cite{BorweinLewis.92}---in our analogy, connectedness means that any trip between two points of~$A$ can be achieved with finitely many stopovers at some points in~$A$.

\begin{lemma}\label{le:connectingPoint}
   Let $P=\mu\otimes\nu$ and let $A,B\subset\X\times\Y$ be Borel sets with $P(B)>0$ and $P(A)=1$. There are Borel sets $\X_{0}\subset\X$ and $\Y_{0}\subset \Y$ with $\mu(\X_{0})=\nu(\Y_{0})=1$ such that setting $A_{0}:=A\cap (\X_{0}\times\Y_{0})$ and  $B_{0}:=B\cap (\X_{0}\times\Y_{0})$, there exists a point 
\begin{center}   
   $(x_{*},y_{*})\in B_{0}$ such that $(x,y_{*}), (x_{*},y)\in A_{0}$ for any $(x,y)\in \X_{0}\times\Y_{0}$.
 \end{center}  
\end{lemma} 

\begin{proof}
  Let $C_{x}= \{y: (x,y)\in C\}$ denote the section of a set $C\subset\X\times\Y$ at $x\in\X$, and analogously for $y\in\Y$. Let $\X_{1}=\{x\in\X: \nu(A_{x})=1\}$. In view of Fubini's theorem, $P(A)=1$ implies $\mu(\X_{1})=1$. Similarly, $P(B)>0$ implies that $\{x\in\X: \nu(B_{x})>0\}$ has positive $\mu$-measure. In particular, there exists a point $x_{*}\in\X_{1}$ with $\nu(B_{x_{*}})>0$. 
  
  Next, let $\Y_{0}=\{y\in\Y: \mu(A_{y})=1\} \cap A_{x_{*}}$. Then again $\nu(\Y_{0})=1$, and in particular there exists a point $y_{*}\in\Y_{0}\cap B_{x_{*}}$. Moreover, the set $\X_{0}:=\X_{1}\cap A_{y_{*}}$ satisfies $\mu(\X_{0})=1$. By passing to Borel subsets of full measure, we may assume that $\X_{0},\Y_{0}$ are themselves Borel.
  
  Writing $A_{0}=A\cap (\X_{0}\times\Y_{0})$ and  $B_{0}=B\cap (\X_{0}\times\Y_{0})$, we have by construction  that $(x_{*},y_{*})\in B_{0}$ satisfies $(x_{*},y)\in A_{0}$ for all $y\in\Y_{0}$ and $(x,y_{*})\in A_{0}$ for all $x\in\X_{0}$. 
\end{proof} 

\begin{proof}[Proof of Proposition~\ref{pr:weakIsStrong} $(ii)\Rightarrow(iii)$.]
  Let $Z,\Omega_{1},\Omega_{0}$ be as in Definition~\ref{de:weaklyCyclInv}. Our aim is to find Borel functions $\varphi:\X\to(0,\infty)$ and $\psi:\Y\to(0,\infty)$ such that $Z'(x,y):=\varphi(x)\psi(y)$ defines a version of the density $d\pi/dR$. It is sufficient to construct $\varphi$ on a Borel set $\X_{0}\subset\X$ of full marginal measure, as we may then extend~$\varphi$ by setting~$\varphi=1$ on~$\X\setminus \X_{0}$, and similarly for~$\psi$. In particular, we may assume that $\proj_{\X}\Omega_{1}=\X$ and $\proj_{\Y}\Omega_{1}=\Y$.
  
  Noting that $P(\Omega_{1})>0$ due to  $\pi\ll P$, we can apply Lemma~\ref{le:connectingPoint} to~$\Omega_{1}$ and~$\Omega_{0}$, and in view of the above observation, we may assume that $\X_{0}=\X$ and $\Y_{0}=\Y$ in its assertion. We then obtain a point
  \begin{center}   
   $(x_{*},y_{*})\in \Omega_{1}$ such that $(x,y_{*}), (x_{*},y)\in\Omega_{0}$ for any $(x,y)\in \X\times\Y$.
 \end{center}  
  Define $\varphi(x_{*}):=a>0$ as an arbitrary number and $\psi(y_{*}):=Z(x_{*},y_{*})/\varphi(x_{*})$. 
 Given $(x,y)\in\X\times\Y$, we have $(x_{*},y)\in\Omega_{0}$ and $(x,y_{*})\in\Omega_{0}$, allowing us to define
  $$
    \psi(y):=Z(x_{*},y)/\varphi(x_{*}) \in (0,\infty), \quad \varphi(x):=Z(x,y_{*})/\psi(y_{*})\in (0,\infty).
  $$
  The fact that~$Z$ is Borel readily implies that $\varphi,\psi$ are Borel.
  Define $Z'(x,y):=\varphi(x)\psi(x)$ for $(x,y)\in\X\times\Y$. Clearly $Z'$ is Borel and takes values in~$(0,\infty)$.  
  Let $(x,y)\in\Omega_{1}$. Then $(x,y_{*}), (x_{*},y)\in\Omega_{0}$ and weak cyclical invariance yields
  $$
    \varphi(x)\psi(y)=\frac{Z(x,y_{*})}{\psi(y_{*})}\frac{Z(x_{*},y)}{\varphi(x_{*})} = \frac{Z(x,y_{*}) Z(x_{*},y)} {Z(x_{*},y_{*})} = Z(x,y).
  $$
  That is, $Z'=Z$ on $\Omega_{1}$. To prove the same relation on $\Omega_{0}$, consider $(x,y)\in \Omega_{0}$. There are  $x'\in\X$ and $y'\in\Y$ such that 
  $z_{1}:=(x,y')$ and $z_{2}:=(x',y)$ are in~$\Omega_{1}$, and of course we also have $z_{3}:=(x_{*},y_{*})\in\Omega_{1}$. 
  Thus  the already established fact that $Z'=Z$ on $\Omega_{1}$ implies 
  $$
    \varphi(x)\psi(y)=\frac{\varphi(x)\psi(y') ~ \varphi(x')\psi(y) ~ \varphi(x_{*})\psi(y_{*})} {\varphi(x')\psi(y_{*}) ~ \varphi(x_{*})\psi(y') } = \frac{Z(x,y') Z(x',y) Z(x_{*},y_{*})}{Z(x',y_{*})Z(x_{*},y')}.
  $$
  On the other hand, $\bz_{1}=(x,y)$ and $\bz_{2}=(x',y_{*})$ and $\bz_{3}=(x_{*},y')$ are all in~$\Omega_{0}$, so that the invariance yields
  $$
    0< Z(x,y') Z(x',y) Z(x_{*},y_{*}) =  Z(x,y)Z(x',y_{*})Z(x_{*},y').
  $$
  As a result, $Z'(x,y)=\varphi(x)\psi(y)=Z(x,y)$, showing $Z'=Z$ on $\Omega_{0}$. Recalling $R(\Omega_{0})$=1, it follows that $Z'$ is again a version of $d\pi/dR$.
\end{proof}

\section{Proof of Main Results and Ramifications}\label{se:proofOfMain}

For ease of reference, we first summarize some known results.

\begin{lemma}\label{le:uniquenessAndMinimizer}
  Let $(\mu,\nu)\in\cP(\X)\times\cP(\Y)$, $R\in\cP(\X\times\Y)$ and $R\sim P:=\mu\otimes\nu$.
  \begin{enumerate}
  \item If the static Schr\"odinger bridge problem~\eqref{eq:SB} is finite, it admits a unique minimizer~$\pi\in\Pi(\mu,\nu)$. Moreover, $(\pi,R)$ is cyclically invariant.
  \item
  Let $\pi\in\Pi(\mu,\nu)$. If $(\pi,R)$ is cyclically invariant and~\eqref{eq:SB} is finite, then $\pi$ is its minimizer.
    \item
  There exists at most one $\pi\in\Pi(\mu,\nu)$ such that $(\pi,R)$ is cyclically invariant.
\end{enumerate}  
\end{lemma}

\begin{proof}
  Recall that cyclical invariance of $(\pi,R)$ is equivalent a factorization of the density~$d\pi/dR$ into strictly positive functions; cf.~Proposition~\ref{pr:weakIsStrong} or~\cite{BorweinLewis.92}. Taking that into account, (i)~and~(ii) can be found in \cite[Theorem~2.1]{Nutz.20} in the stated generality. (The original results are due to \cite{BorweinLewis.92, BorweinLewisNussbaum.94, Csiszar.75, FollmerGantert.97, RuschendorfThomsen.97}, among others.) Finally, (iii) follows from~(ii), as was also noted in \cite[Corollary~2.9]{Nutz.20}: if  $\pi,\pi'\in\Pi(\mu,\nu)$ have positive densities $d\pi/dR$, $d\pi'/dR$ admitting factorizations, then $d\pi/d\pi'$ also admits a factorization and now~(ii), applied with $\pi'$ as reference measure, implies that~$\pi$ is the unique minimizer of $H(\cdot|\pi')$. As $\pi'\in\Pi(\mu,\nu)$ is itself a coupling, this minimizer is~$\pi'$.
\end{proof}

\begin{proof}[Proof of Theorems~\ref{th:wellposednessEOT}, \ref{th:stabilityEOT}, \ref{th:wellposedness} and \ref{th:stability}.]
  Given data as in Theorem~\ref{th:stability}, the sequences $(\mu_{n})$ and $(\nu_{n})$ are tight, which readily implies the tightness of~$(\pi_{n})$; cf.\  \cite[Lemma~4.4, p.\,44]{Villani.09}. In view of Propositions~\ref{pr:limitEquiv} and~\ref{pr:weakIsStrong}, any cluster point $\pi$ is such that $(\pi,R)$ is cyclically invariant. The uniqueness of cyclically invariant couplings, see Lemma~\ref{le:uniquenessAndMinimizer}\,(iii), shows that all cluster points coincide and hence that the original sequence $(\pi_{n})$ converges. This proves Theorem~\ref{th:stability}.
  
  To deduce Theorem~\ref{th:stabilityEOT} from Theorem~\ref{th:stability}, we choose the reference measure as in~\eqref{eq:defR}; i.e.,
  $$
    \frac{dR_{n}}{dP_{n}} = a_{n} e^{-c_{n}/\eps_{n}}, \qquad \frac{dR}{dP} = a e^{-c/\eps},
  $$
  where $a_{n},a$ are the normalizing constants. Combining the uniform convergence $c_{n}/\eps_{n}\to c/\eps$ on bounded sets with the continuity of~$c$, we see that~\eqref{eq:mainCond} holds, for instance with $\alpha_{n}=a_{n}/a$. %
  
  The uniqueness part of Theorem~\ref{th:wellposedness}, as well as its last assertion, are stated in Lemma~\ref{le:uniquenessAndMinimizer}. To deduce existence from 
  Theorem~\ref{th:stability}, we consider the constant marginals $(\mu_{n},\nu_{n}):=(\mu,\nu)$ and define approximating reference measures $R_{n}$ via
  $$
    \frac{dR_{n}}{dP} = a_{n} \bigg(\frac{dR}{dP} \vee \frac{1}{n}\bigg),
  $$
  where $a_{n}$ is the (finite) normalizing constant. As $\frac{dR}{dP}$ is continuous and positive, hence bounded away from zero on small balls, the condition~\eqref{eq:mainCond} is satisfied with $P_{n}=P$ and a function~$o(1)$ independent of~$n$. The static Schr\"odinger bridge problem~\eqref{eq:SB} for $R_{n}$ falls into the classical setting of Lemma~\ref{le:uniquenessAndMinimizer}\,(i) because the product coupling $\pi_{0}:=\mu\otimes\nu$ satisfies $H(\pi_{0}|R_{n})<\infty$. In particular, the associated cyclically invariant couplings $\pi_{n}\in\Pi(\mu,\nu)$ exist, and now  Theorem~\ref{th:stability} implies the existence of $\pi\in\Pi(\mu,\nu)$ such that $(\pi,R)$ is cyclically invariant.
  Finally, Theorem~\ref{th:wellposednessEOT} is a direct consequence of Theorem~\ref{th:wellposedness} via~\eqref{eq:defR}. 
\end{proof}

\begin{proof}[Proof of Corollary~\ref{co:Schr\"odingerEqn}]
  Define $R\in\cP(\X\times\Y)$ by $dR/dP:=f$ and let~$\pi$ be as in Theorem~\ref{th:wellposedness}. As $(\pi,R)$ is cyclically invariant, a version of the relative density admits a factorization $d\pi/dR(x,y)=\varphi(x)\psi(y)$ into positive Borel functions; cf.\ Proposition~\ref{pr:weakIsStrong}. The fact that~$\pi$ has marginals $(\mu,\nu)$ then translates to the fact that $(\varphi,\psi)$ solve the Schr\"odinger system. Uniqueness of $\varphi,\psi$ up to a constant follows from the uniqueness of~$\pi$ (here the fact that $R\sim P$ is particularly important).
\end{proof} 

\begin{remark}\label{rk:doubling}
  Assumption~\ref{as:LebesgueDiff} can be replaced by the assumption that the marginals $(\X,\mu)$ and $(\Y,\nu)$ satisfy the so-called doubling property. The latter assumption is structurally different as it refers to the specific measures rather than the metric spaces.
  Indeed, let $(\X,\mu)$ be doubling; i.e., there exist $C>0$ such that
  $$
    \mu(B_{2r}(x))\leq C \mu(B_{r}(x))
  $$
  for any ball $B_{r}(x)\subset\X$. This ensures that differentiation of measures (in the sense of Assumption~\ref{as:LebesgueDiff}) holds for measures $\rho\ll\mu$; cf.\ \cite[Theorem~1.8, p.\,4]{Heinonen.01}. In particular, Lemma~\ref{le:conditionalDiffOfMeas} holds as stated, and then so does Proposition~\ref{pr:limitAbsCont}. If $(\Y,\nu)$ is also doubling, then so is $(\X\times\Y,P)$ where $P=\mu\otimes\nu$, showing that differentiation wrt.~$P$ holds for measures $\rho\ll P$, in particular for $\rho:=R\sim P$ in the context of Proposition~\ref{pr:limitEquiv}. As $dR/dP>0$, it follows that differentiation also holds wrt.~$R$. This ensures that the proof of Proposition~\ref{pr:limitEquiv} remains valid, and hence the main results.
\end{remark} 

\subsection{Instability for Discontinuous Costs}

We show that stability of entropic optimal transport fails as soon as the cost function has an ``essential'' discontinuity. To see why the qualifier is necessary, consider a cost function~$c$ of the form
$$
  c(x,y)=\hat{c}(x,y)-f(x)-g(y)
$$
for some possibly discontinuous functions $f:\X\to\R$, $g:\Y\to\R$ and a continuous function~$\hat{c}$. It is immediate from Definition~\ref{de:cyclInvEOT} that the marginal functions $f$ and $g$ do not affect cyclical invariance. The cost~$c$ is equivalent to~$\hat{c}$ from this perspective; in terms of optimal transport, the interpretation is that $f,g$ change the cost of any coupling by the same constant, and therefore do not change the optimizer. As a result, a discontinuity in~$c$ can only be relevant for stability if~$c$ cannot be written in the above form with a continuous~$\hat{c}$, which is the condition below.

\begin{proposition}\label{pr:discont}
  Let $c$ be a cost function such that $(x,y)\mapsto c(x,y)-f(x)-g(y)$ is discontinuous for any functions $f:\X\to\R$ and $g:\Y\to\R$. Then the stability of entropic optimal transport fails for~$c$. That is, there are marginals $(\mu_{n},\nu_{n})\to(\mu,\nu)$ and $(c,1)$-cyclically invariant couplings $\pi_{n}\in\Pi(\mu_{n},\nu_{n})$ with $\pi_{n}\to\pi\in\Pi(\mu,\nu)$ where $\pi$ is not $(c,1)$-cyclically invariant.
\end{proposition} 

\begin{proof}
Fix arbitrary $(x_{0},y_{0})\in\X\times\Y$. As $(x,y)\mapsto c(x,y) - c(x,y_{0}) - c(x_{0},y)$ is discontinuous, there is a sequence $(x_{n},y_{n})\to(x_{\infty},y_{\infty})$ such that
\begin{equation}\label{eq:discont1}
  c(x_{n},y_{n})- c(x_{n},y_{0}) - c(x_{0},y_{n}) \: \not\to \:  c(x_{\infty},y_{\infty})- c(x_{\infty},y_{0}) - c(x_{0},y_{\infty}).
\end{equation}
Consider the marginals $\mu_{n}=(\delta_{x_{0}}+\delta_{x_{n}})/2$ and $\nu_{n}=(\delta_{y_{0}}+\delta_{y_{n}})/2$. Let $\pi_{n}\in\Pi(\mu_{n},\nu_{n})$ be $(c,1)$-cyclically invariant; that is,
\begin{equation}\label{eq:discont2}
  \frac{\pi_{n}(x_{0},y_{0})\pi_{n}(x_{n},y_{n})} {\pi_{n}(x_{0},y_{n})\pi_{n}(x_{n},y_{0})} = \exp\big[c(x_{0},y_{n})+c(x_{n},y_{0})  -c(x_{0},y_{0})-c(x_{n},y_{n})\big].
\end{equation}
 After passing to a subsequence, $\pi_{n}$ converge weakly to a coupling $\pi\in\Pi(\mu,\nu)$ of $\mu:=(\delta_{x_{0}}+\delta_{x_{\infty}})/2$ and $\nu:=(\delta_{y_{0}}+\delta_{y_{\infty}})/2$. Suppose for contradiction that~$\pi$ is cyclically invariant, then
\begin{equation}\label{eq:discont3}
  \frac{\pi(x_{0},y_{0})\pi(x_{\infty},y_{\infty})} {\pi(x_{0},y_{\infty})\pi(x_{\infty},y_{0})} 
  = \exp\big[c(x_{0},y_{\infty})+c(x_{\infty},y_{0})  -c(x_{0},y_{0})-c(x_{\infty},y_{\infty})\big].
\end{equation}
As the left-hand side of~\eqref{eq:discont2} converges to the left-hand side of~\eqref{eq:discont3}, the convergence of the right-hand sides follows, contradicting~\eqref{eq:discont1}.
\end{proof} 

\newcommand{\dummy}[1]{}


\begin{thebibliography}{10}

\bibitem{AlvarezJaakkola.18}
D.~Alvarez-Melis and T.~Jaakkola.
\newblock Gromov-{W}asserstein alignment of word embedding spaces.
\newblock In {\em Proceedings of the 2018 Conference on Empirical Methods in
  Natural Language Processing}, pages 1881--1890, 2018.

\bibitem{WGAN.17}
M.~Arjovsky, S.~Chintala, and L.~Bottou.
\newblock {W}asserstein generative adversarial networks.
\newblock volume~70 of {\em Proceedings of Machine Learning Research}, pages
  214--223, 2017.

\bibitem{BackhoffBeiglbockConforti.21}
J.~Backhoff-Veraguas, M.~Beiglb{\"o}ck, and G.~Conforti.
\newblock A non-linear monotonicity principle and application to the
  {S}chr{\"o}dinger problem.
\newblock {\em Preprint arXiv:2101.09975v1}, 2021.

\bibitem{BackhoffBeiglbockPammer.19}
J.~Backhoff-Veraguas, M.~Beiglb\"{o}ck, and G.~Pammer.
\newblock Existence, duality, and cyclical monotonicity for weak transport
  costs.
\newblock {\em Calc. Var. Partial Differential Equations}, 58(6):Paper No. 203,
  28, 2019.

\bibitem{BeiglbockCoxHuesmann.14}
M.~Beiglb{\"o}ck, A.~M.~G. Cox, and M.~Huesmann.
\newblock Optimal transport and {S}korokhod embedding.
\newblock {\em Invent. Math.}, 208(2):327--400, 2017.

\bibitem{BeiglbockGoldsternMareschSchachermayer.09}
M.~Beiglb{\"o}ck, M.~Goldstern, G.~Maresch, and W.~Schachermayer.
\newblock Optimal and better transport plans.
\newblock {\em J. Funct. Anal.}, 256(6):1907--1927, 2009.

\bibitem{BeiglbockJuillet.12}
M.~Beiglb{\"o}ck and N.~Juillet.
\newblock On a problem of optimal transport under marginal martingale
  constraints.
\newblock {\em Ann. Probab.}, 44(1):42--106, 2016.

\bibitem{BeiglbockNutzStebegg.21}
M.~Beiglb\"{o}ck, M.~Nutz, and F.~Stebegg.
\newblock Fine properties of the optimal {S}korokhod embedding problem.
\newblock {\em J. Eur. Math. Soc. (JEMS)}, 24(4):1389--1429, 2022.

\bibitem{Berman.20}
R.~J. Berman.
\newblock The {S}inkhorn algorithm, parabolic optimal transport and geometric
  {M}onge-{A}mp{\`e}re equations.
\newblock {\em Numer. Math.}, 145(4):771--836, 2020.

\bibitem{BerntonGhosalNutz.21}
E.~Bernton, P.~Ghosal, and M.~Nutz.
\newblock Entropic optimal transport: Geometry and large deviations.
\newblock {\em Duke Math. J., to appear}.

\bibitem{Beurling.60}
A.~Beurling.
\newblock An automorphism of product measures.
\newblock {\em Ann. of Math. (2)}, 72:189--200, 1960.

\bibitem{BorweinLewis.92}
J.~M. Borwein and A.~S. Lewis.
\newblock Decomposition of multivariate functions.
\newblock {\em Canad. J. Math.}, 44(3):463--482, 1992.

\bibitem{BorweinLewisNussbaum.94}
J.~M. Borwein, A.~S. Lewis, and R.~D. Nussbaum.
\newblock Entropy minimization, {$DAD$} problems, and doubly stochastic
  kernels.
\newblock {\em J. Funct. Anal.}, 123(2):264--307, 1994.

\bibitem{CarlierDuvalPeyreSchmitzer.17}
G.~Carlier, V.~Duval, G.~Peyr\'{e}, and B.~Schmitzer.
\newblock Convergence of entropic schemes for optimal transport and gradient
  flows.
\newblock {\em SIAM J. Math. Anal.}, 49(2):1385--1418, 2017.

\bibitem{ChernozhukovEtAl.17}
V.~Chernozhukov, A.~Galichon, M.~Hallin, and M.~Henry.
\newblock Monge-{K}antorovich depth, quantiles, ranks and signs.
\newblock {\em Ann. Statist.}, 45(1):223--256, 2017.

\bibitem{ConfortiTamanini.19}
G.~Conforti and L.~Tamanini.
\newblock A formula for the time derivative of the entropic cost and
  applications.
\newblock {\em J. Funct. Anal.}, 280(11):108964, 2021.

\bibitem{Csiszar.75}
I.~Csisz\'{a}r.
\newblock {$I$}-divergence geometry of probability distributions and
  minimization problems.
\newblock {\em Ann. Probability}, 3:146--158, 1975.

\bibitem{CuturiTeboulVert.19}
M.~Cuturi, O.~Teboul, and J.-P. Vert.
\newblock Differentiable ranking and sorting using optimal transport.
\newblock In {\em Advances in Neural Information Processing Systems},
  volume~32, 2019.

\bibitem{DebSen.19}
N.~Deb and B.~Sen.
\newblock Multivariate rank-based distribution-free nonparametric testing using
  measure transportation.
\newblock {\em Preprint arXiv:1909.08733v1}, 2019.

\bibitem{delBarrioEtAl.18}
E.~del Barrio, J.~A. Cuesta-Albertos, M.~Hallin, and C.~Matr{\'a}n.
\newblock Distribution and quantile functions, ranks and signs in dimension
  {$d$}: A measure transportation approach.
\newblock {\em Ann. Statist.}, 49(2):1139--1165, 2021.

\bibitem{EvansGariepy.15}
L.~C. Evans and R.~F. Gariepy.
\newblock {\em Measure theory and fine properties of functions}.
\newblock Textbooks in Mathematics. CRC Press, Boca Raton, FL, revised edition,
  2015.

\bibitem{Falconer.86}
K.~J. Falconer.
\newblock {\em The geometry of fractal sets}, volume~85 of {\em Cambridge
  Tracts in Mathematics}.
\newblock Cambridge University Press, Cambridge, 1986.

\bibitem{Federer.69}
H.~Federer.
\newblock {\em Geometric measure theory}, volume 153 of {\em Grundlehren der
  mathematischen Wissenschaften}.
\newblock Springer-Verlag, New York, 1969.

\bibitem{Follmer.88}
H.~F\"{o}llmer.
\newblock Random fields and diffusion processes.
\newblock In {\em \'{E}cole d'\'{E}t\'{e} de {P}robabilit\'{e}s de
  {S}aint-{F}lour {XV}--{XVII}, 1985--87}, volume 1362 of {\em Lecture Notes in
  Math.}, pages 101--203. Springer, Berlin, 1988.

\bibitem{FollmerGantert.97}
H.~F\"{o}llmer and N.~Gantert.
\newblock Entropy minimization and {S}chr\"{o}dinger processes in infinite
  dimensions.
\newblock {\em Ann. Probab.}, 25(2):901--926, 1997.

\bibitem{Fortet.40}
R.~Fortet.
\newblock R\'{e}solution d'un syst\`eme d'\'{e}quations de {M}.
  {S}chr\"{o}dinger.
\newblock {\em J. Math. Pures Appl.}, 19:83--105, 1940.

\bibitem{GenevayPeyreCuturi.18}
A.~Genevay, G.~Peyr{\'e}, and M.~Cuturi.
\newblock Learning generative models with {S}inkhorn divergences.
\newblock In {\em Proceedings of the 21st International Conference on
  Artificial Intelligence and Statistics}, PMLR, pages 1608--1617, 2018.

\bibitem{GhosalSen.19}
P.~Ghosal and B.~Sen.
\newblock Multivariate ranks and quantiles using optimal transport:
  consistency, rates and nonparametric testing.
\newblock {\em Ann. Statist.}, 50(2):1012--1037, 2022.

\bibitem{GigliTamanini.21}
N.~Gigli and L.~Tamanini.
\newblock Second order differentiation formula on {${\mathrm{RCD}}^{*}(K,N)$}
  spaces.
\newblock {\em J. Eur. Math. Soc. (JEMS)}, 23(5):1727--1795, 2021.

\bibitem{HarchaouiLiuPal.20}
Z.~Harchaoui, L.~Liu, and S.~Pal.
\newblock Asymptotics of entropy-regularized optimal transport via chaos
  decomposition.
\newblock {\em Preprint arXiv:2011.08963v1}, 2020.

\bibitem{Heinonen.01}
J.~Heinonen.
\newblock {\em Lectures on analysis on metric spaces}.
\newblock Universitext. Springer-Verlag, New York, 2001.

\bibitem{HobbyPyke.65}
C.~Hobby and R.~Pyke.
\newblock Doubly stochastic operators obtained from positive operators.
\newblock {\em Pacific J. Math.}, 15:153--157, 1965.

\bibitem{Leonard.12}
C.~L\'{e}onard.
\newblock From the {S}chr\"{o}dinger problem to the {M}onge-{K}antorovich
  problem.
\newblock {\em J. Funct. Anal.}, 262(4):1879--1920, 2012.

\bibitem{Leonard.14}
C.~L\'{e}onard.
\newblock A survey of the {S}chr\"{o}dinger problem and some of its connections
  with optimal transport.
\newblock {\em Discrete Contin. Dyn. Syst.}, 34(4):1533--1574, 2014.

\bibitem{Leonard.19}
C.~L\'{e}onard.
\newblock Revisiting {F}ortet's proof of existence of a solution to the
  {S}chr{\"o}dinger system.
\newblock {\em Preprint arXiv:1904.13211v1}, 2019.

\bibitem{McCann.95}
R.~J. McCann.
\newblock Existence and uniqueness of monotone measure-preserving maps.
\newblock {\em Duke Math. J.}, 80(2):309--323, 1995.

\bibitem{MenaWeed.19}
G.~Mena and J.~Niles-Weed.
\newblock Statistical bounds for entropic optimal transport: sample complexity
  and the central limit theorem.
\newblock In {\em Advances in Neural Information Processing Systems 32}, pages
  4541--4551. 2019.

\bibitem{Mikami.02}
T.~Mikami.
\newblock Optimal control for absolutely continuous stochastic processes and
  the mass transportation problem.
\newblock {\em Electron. Comm. Probab.}, 7:199--213, 2002.

\bibitem{Mikami.04}
T.~Mikami.
\newblock Monge's problem with a quadratic cost by the zero-noise limit of
  {$h$}-path processes.
\newblock {\em Probab. Theory Related Fields}, 129(2):245--260, 2004.

\bibitem{Nutz.20}
M.~Nutz.
\newblock {\em Introduction to Entropic Optimal Transport}.
\newblock Lecture notes, Columbia University, 2021.
\newblock
  \url{https://www.math.columbia.edu/~mnutz/docs/EOT_lecture_notes.pdf}.

\bibitem{NutzWiesel.21}
M.~Nutz and J.~Wiesel.
\newblock Entropic optimal transport: Convergence of potentials.
\newblock {\em Probab. Theory Related Fields, to appear}.

\bibitem{Pal.19}
S.~Pal.
\newblock On the difference between entropic cost and the optimal transport
  cost.
\newblock {\em Preprint arXiv:1905.12206v1}, 2019.

\bibitem{Parthasarathy.67}
K.~R. Parthasarathy.
\newblock {\em Probability measures on metric spaces}.
\newblock Probability and Mathematical Statistics, No. 3. Academic Press, New
  York, 1967.

\bibitem{CuturiPeyre.19}
G.~Peyr{\'e} and M.~Cuturi.
\newblock Computational optimal transport: With applications to data science.
\newblock {\em Foundations and Trends in Machine Learning}, 11(5-6):355--607,
  2019.

\bibitem{Pfanzagl.79}
J.~Pfanzagl.
\newblock Conditional distributions as derivatives.
\newblock {\em Ann. Probab.}, 7(6):1046--1050, 1979.

\bibitem{Preiss.79}
D.~Preiss.
\newblock Gaussian measures and covering theorems.
\newblock {\em Comment. Math. Univ. Carolin.}, 20(1):95--99, 1979.

\bibitem{Preiss.81}
D.~Preiss.
\newblock Gaussian measures and the density theorem.
\newblock {\em Comment. Math. Univ. Carolin.}, 22(1):181--193, 1981.

\bibitem{RubnerTomasiGuibas.00}
Y.~Rubner, C.~Tomasi, and L.~J. Guibas.
\newblock The earth mover's distance as a metric for image retrieval.
\newblock {\em Int. J. Comput. Vis.}, 40:99--121, 2000.

\bibitem{RuschendorfThomsen.93}
L.~R\"{u}schendorf and W.~Thomsen.
\newblock Note on the {S}chr\"{o}dinger equation and {$I$}-projections.
\newblock {\em Statist. Probab. Lett.}, 17(5):369--375, 1993.

\bibitem{RuschendorfThomsen.97}
L.~R\"{u}schendorf and W.~Thomsen.
\newblock Closedness of sum spaces and the generalized {S}chr\"{o}dinger
  problem.
\newblock {\em Teor. Veroyatnost. i Primenen.}, 42(3):576--590, 1997.

\bibitem{Schmitzer.19}
B.~Schmitzer.
\newblock Stabilized sparse scaling algorithms for entropy regularized
  transport problems.
\newblock {\em SIAM J. Sci. Comput.}, 41(3):A1443--A1481, 2019.

\bibitem{Villani.09}
C.~Villani.
\newblock {\em Optimal transport, old and new}, volume 338 of {\em Grundlehren
  der Mathematischen Wissenschaften}.
\newblock Springer-Verlag, Berlin, 2009.

\end{thebibliography}
\end{document}